\documentclass[11pt]{article}
\usepackage{amsmath,amsthm,amsfonts,comment,array}

\DeclareMathOperator{\Hom}{Hom}
\DeclareMathOperator{\Com}{Com}

\DeclareMathOperator{\1}{id}

\newcommand{\NN}{\mathbb{N}}
\newcommand{\RR}{\mathbb{R}}

\newcommand{\EEnd}{\mathcal End}
\newcommand{\EE}{\mathcal E}
\newcommand{\bul}{\bullet}

\renewcommand{\=}{:=}
\renewcommand{\t}{\otimes}
\renewcommand{\:}{\colon}

\newcommand{\m}{\overset{\circ}{\mu}}

\newtheorem{thm}{Theorem}[section]

 \newtheorem{lemma}[thm]{Lemma}

\theoremstyle{definition}
 \newtheorem{defn}[thm]{Definition}

\theoremstyle{definition}

\theoremstyle{definition}
 \newtheorem{rem}[thm]{Remark}

\numberwithin{equation}{section}
\numberwithin{table}{section}

\begin{document}
\title{\LARGE\bf  Dynamical deformations of three-dimensional \\ Lie algebras in  Bianchi classification over  the \\harmonic oscillator}
\author{\Large Eugen Paal and J\"{u}ri Virkepu 
}
\date{}
\maketitle
\thispagestyle{empty}
\begin{abstract}
Operadic Lax representations for the  harmonic oscillator are used to construct the dynamical deformations of three-dimensional (3D) real Lie algebras in the Bianchi classification. It is shown that the energy conservation of the harmonic oscillator is related to the Jacobi identities of the dynamically deformed algebras. Based on this observation, it is proved  that the dynamical deformations of 3D real Lie algebras in the Bianchi classification over the harmonic oscillator are Lie algebras.
\end{abstract}

\section{Introduction}

In Hamiltonian formalism, a mechanical system is described by the
canonical variables $q^i,p_i$ and their time evolution is prescribed
by the Hamiltonian equations
\begin{equation}
\label{ham} \dfrac{dq^i}{dt}=\dfrac{\partial H}{\partial p_i}, \quad
\dfrac{dp_i}{dt}=-\dfrac{\partial H}{\partial q^i}
\end{equation}
By a Lax representation \cite{Lax68,BBT03} of a mechanical system
one means such a pair $(L,M)$ of matrices (linear operators) $L,M$
that the above Hamiltonian system may be represented as the Lax
equation
\begin{equation}
\label{lax} \dfrac{dL}{dt}= ML-LM
\end{equation}
Thus, from the algebraic point of view, mechanical systems can be
described by linear operators, i.e by  linear maps $V\to V$  of a
vector space $V$. As a generalization of this one can pose the
following question \cite{Paal07}: how can the time evolution of the linear operations (multiplications) $V^{\t n}\to V$ be described?

The algebraic operations (multiplications) can be seen as an example
of the \emph{operadic} variables \cite{Ger}. If an operadic system
depends on time one can speak about \emph{operadic dynamics}
\cite{Paal07}. The latter may be introduced by simple and natural
analogy with the Hamiltonian dynamics. In particular, the time
evolution of the operadic variables may be given by the operadic Lax
equation. In Refs. \cite{PV07,PV08,PV08-1}, low-dimensional binary operadic Lax representations for the harmonic oscillator were constructed.

In the present paper, the operadic Lax
representations for the  harmonic oscillator are used to construct the dynamical deformations of three-dimensional (3D) real Lie algebras in the Bianchi classification.
It is shown that the energy conservation of the harmonic oscillator is related to the Jacobi identities of the dynamically deformed algebras.
Based on this observation, it is proved that the dynamical deformations of 3D real Lie algebras in the Bianchi classification over the harmonic oscillator are Lie algebras.

It turns out that four Lie algebras (I, VII, IX, VIII  from the Bianchi classification) remain undeformed (rigid) and all other ones are deformed. However, it is interesting to note that these remain to be Lie algebras over canonical variables $q,p$.  Namely, four of them
(II$\,^{t}$, IV$\,^{t}$, V$\,^{t}$, VI$\,^{t}$) lead straightforwardly to the Jacobi identity, while in other cases
(VII$^{\,t}_a$, III$_{a=1}^{\,t}$, VI$_{a\neq1}^{\,t}$) satisfy the Jacobi identity with  the energy conservation law.

\section{Endomorphism operad and Gerstenhaber brackets}

Let $K$ be a unital associative commutative ring, $V$ be a unital
$K$-module, and $\EE_V^n\= {\EEnd}_V^n\= \Hom(V^{\t n},V)$
($n\in\NN$). For an \emph{operation} $f\in\EE^n_V$, we refer to $n$
as the \emph{degree} of $f$ and often write (when it does not cause
confusion) $f$ instead of $\deg f$. For example, $(-1)^f\= (-1)^n$,
$\EE^f_V\=\EE^n_V$ and $\circ_f\= \circ_n$. Also, it is convenient
to use the \emph{reduced} degree $|f|\= n-1$. Throughout this paper,
we assume that $\t\= \t_K$.

\begin{defn}[endomorphism operad \cite{Ger}]
\label{HG} For $f\t g\in\EE_V^f\t\EE_V^g$ define the \emph{partial
compositions}
\[
f\circ_i g\= (-1)^{i|g|}f\circ(\1_V^{\t i}\t g\t\1_V^{\t(|f|-i)})
\quad \in\EE^{f+|g|}_V,
         \quad 0\leq i\leq |f|
\]
The sequence $\EE_V\= \{\EE_V^n\}_{n\in\NN}$, equipped with the
partial compositions $\circ_i$, is called the \emph{endomorphism
operad} of $V$.
\end{defn}

\begin{defn}[total composition \cite{Ger}]
The \emph{total composition}
$\bul\:\EE^f_V\t\EE^g_V\to\EE^{f+|g|}_V$ is defined by
\[
f\bul g\= \sum_{i=0}^{|f|}f\circ_i g\quad \in \EE_V^{f+|g|}, \quad |\bul|=0
\]
The pair $\Com\EE_V\= \{\EE_V,\bul\}$ is called the \emph{composition
algebra} of $\EE_V$.
\end{defn}

\begin{defn}[Gerstenhaber brackets \cite{Ger}]
The  \emph{Gerstenhaber brackets} $[\cdot,\cdot]$ are defined in
$\Com\EE_V$ as a graded commutator by
\[
[f,g]\= f\bul g-(-1)^{|f||g|}g\bul f=-(-1)^{|f||g|}[g,f],\quad
|[\cdot,\cdot]|=0
\]
\end{defn}

The \emph{commutator algebra} of $\Com \EE_V$ is denoted as
$\Com^{-}\!\EE_V\= \{\EE_V,[\cdot,\cdot]\}$. One can prove (e.g
\cite{Ger}) that $\Com^-\!\EE_V$ is a \emph{graded Lie algebra}. The
Jacobi identity reads
\[
(-1)^{|f||h|}[f,[g,h]]+(-1)^{|g||f|}[g,[h,f]]+(-1)^{|h||g|}[h,[f,g]]=0
\]

\section{Operadic Lax equation and harmonic oscillator}

Assume that $K\= \RR$ or $K\= \mathbb{C}$ and operations are
differentiable. Dynamics in operadic systems (operadic dynamics) may
be introduced by

\begin{defn}[operadic Lax pair \cite{Paal07}]
Allow a classical dynamical system to be described by the
Hamiltonian system \eqref{ham}. An \emph{operadic Lax pair} is a
pair $(\mu,M)$ of operations $\mu,M\in\EE_V$, such that the
Hamiltonian system  (\ref{ham}) may be represented as the
\emph{operadic Lax equation}
\[
\frac{d\mu}{dt}=[M,\mu]\= M\bul \mu-(-1)^{|M||\mu|}\mu \bul M
\]
The pair $(\mu,M)$ is also called an \emph{operadic Lax representations} of/for Hamiltonian system \eqref{ham}.
In this paper we assume that $|M|=0$.
\end{defn}

\begin{rem}
Evidently the degree constraints $|M|=|\mu|=0$ give rise to the ordinary
Lax equation (\ref{lax}) \cite{Lax68,BBT03}. If $|\mu|\neq 0$ then the Gerstenhaber brackets do not coincide with the usual  commutator 
(see Section \ref{sec:evolution} for details).
\end{rem}

The Hamiltonian of the harmonic oscillator is
\[
H(q,p)=\frac{1}{2}(p^2+\omega^2q^2)
\]
Thus, the Hamiltonian system of the harmonic oscillator reads
\begin{equation}
\label{eq:h-osc} \frac{dq}{dt}=\frac{\partial H}{\partial p}=p,\quad
\frac{dp}{dt}=-\frac{\partial H}{\partial q}=-\omega^2q
\end{equation}
If $\mu$ is a linear algebraic operation we can use the above
Hamilton equations to obtain
\[
\dfrac{d\mu}{dt} =\dfrac{\partial\mu}{\partial
q}\dfrac{dq}{dt}+\dfrac{\partial\mu}{\partial p}\dfrac{dp}{dt}
=p\dfrac{\partial\mu}{\partial
q}-\omega^2q\dfrac{\partial\mu}{\partial p}
 =[M,\mu]
\]
Therefore, we get the following linear partial differential equation
for $\mu(q,p)$:
\begin{equation}
\label{eq:diff}
p\dfrac{\partial\mu}{\partial
q}-\omega^2q\dfrac{\partial\mu}{\partial p}=[M,\mu]
\end{equation}
By integrating \eqref{eq:diff} one can get collections of operations called the
\emph{operadic (Lax representations for) harmonic oscillator}. Since the general solution of the partial differential equations depends on arbitrary functions, these representations are not uniquely determined.

\section{Evolution of binary algebras}
\label{sec:evolution}

Let $A\= \{V,\mu\}$ be a binary algebra with an operation 
$xy\=\mu(x\t y)$, i.e $|\mu|=1$. Assume that $|M|=0$. We require that 
$\mu=\mu(q,p)$ so that $(\mu,M)$ is an operadic Lax pair, i.e the Hamiltonian system \eqref{eq:h-osc} of the harmonic oscillator may be written as  the operadic Lax equation
\[
\dot{\mu}=[M,\mu]\=  M\bul\mu-\mu\bul M,\quad |\mu|=1,\quad |M|=0
\]
Let $x,y\in V$. Assuming that $|M|=0$ and $|\mu|=1$ we have
\begin{align*}
M\bul\mu &=\sum_{i=0}^0M\circ_i\mu
=M\circ_0\mu=M\circ\mu\\
\mu\bul M &=\sum_{i=0}^1\mu\circ_i M =\mu\circ_0 M+\mu\circ_1
M=\mu\circ(M\t\1_V)+\mu\circ(\1_V\t M)
\end{align*}
Thus we can see that since $|\mu|=1$ the Gerstenhaber brackets of 
$\mu$ and $M$ do not coincide with the common commutator bracketing that is used in the case of the ordinary Lax representations. 
Using the above formulae, we have
\[
\dfrac{d}{dt}(xy)=M(xy)-(Mx)y-x(My)
\]
Let $\dim V=n$. In a basis $\{e_1,\ldots,e_n\}$ of $V$,  the
structure constants $\mu_{jk}^i$ of $A$ are defined by
\[
\mu(e_j\t e_k)\=  \mu_{jk}^i e_i,\quad j,k=1,\ldots,n
\]
In particular,
\[
\dfrac{d}{dt}(e_je_k)=M(e_je_k)-(Me_j)e_k-e_j(Me_k)
\]
By denoting $Me_i\=  M_i^se_s$, it follows that
\[
\dot{\mu}_{jk}^i
=\mu_{jk}^sM_s^i-M_j^s\mu_{sk}^i-M_k^s\mu_{js}^i,\quad
i,j,k=1,\ldots, n
\]

\section{3D binary anti-commutative operadic Lax representations for harmonic oscillator}

\begin{lemma}
\label{lemma:harmonic3} Matrices
\[
L\=\begin{pmatrix}
    p & \omega q & 0 \\
    \omega q & -p & 0 \\
    0 & 0 & 1 \\
  \end{pmatrix},\quad
M\=\frac{\omega}{2}
\begin{pmatrix}
    0 & -1 &0\\
1 & 0 & 0\\
0 & 0 & 0
  \end{pmatrix}
\]
give a 3-dimensional Lax representation for the harmonic oscillator.
\end{lemma}

\begin{defn}
For the harmonic oscillator, define its auxiliary functions $A_\pm$
of canonical variables $p,q$ by
\begin{equation}
\label{eq:def_A}
A_+^2+A_-^2=2\sqrt{2H},\quad
A_+^2-A_-^2=2p,\quad
A_+A_-=\omega q
\end{equation}
\end{defn}

\begin{rem}
Note that the second and the third relations imply the first one in \eqref{eq:def_A}. 
\end{rem}

\begin{thm}[\cite{PV08-1}]
\label{thm:main}
Let $C_{\nu}\in\mathbb{R}$ ($\nu=1,\ldots,9$) be
arbitrary real--valued parameters, such that
\begin{equation}
\label{eq:cond} C_2^2+C_3^2+C_5^2+C_6^2+C_7^2+C_8^2\neq0
\end{equation}
Let $M$ be defined as in Lemma \ref{lemma:harmonic3} and the binary anti-commutative multiplication $\mu:V\otimes V\to V$ is given by coordinates
\begin{equation}\label{eq:theorem}
\begin{cases}
\mu_{11}^{1}=\mu_{22}^{1}=\mu_{33}^{1}=\mu_{11}^{2}=\mu_{22}^{2}=\mu_{33}^{2}=\mu_{11}^{3}=\mu_{22}^{3}=\mu_{33}^{3}=0\\
\mu_{23}^{1}=-\mu_{32}^{1}=C_2p-C_3\omega q-C_4\\
\mu_{13}^{2}=-\mu_{31}^{2}=C_2p-C_3\omega q+C_4\\
\mu_{31}^{1}=-\mu_{13}^{1}=C_2\omega q+C_3p-C_1\\
\mu_{23}^{2}=-\mu_{32}^{2}=C_2\omega q+C_3p+C_1\\
\mu_{12}^{1}=-\mu_{21}^{1}=C_5A_++C_6A_-\\
\mu_{12}^{2}=-\mu_{21}^{2}=C_5A_--C_6A_+\\
\mu_{13}^{3}=-\mu_{31}^{3}=C_7A_++C_8A_-\\
\mu_{23}^{3}=-\mu_{32}^{3}=C_7A_--C_8A_+\\
\mu_{12}^{3}=-\mu_{21}^{3}=C_9
\end{cases}
\end{equation}
Then $(\mu,M)$ is an operadic Lax pair for the harmonic oscillator.
\end{thm}


\section{Initial conditions and dynamical deformations}

It seems attractive to specify the coefficients $C_{\nu}$ in Theorem \ref{thm:main} by the
initial conditions
\[
\left. \mu\right|_{t=0}=\m{}_,\quad
\left.p\right|_{t=0}
=p_0,\quad \left. q\right|_{t=0}=0
\]
The latter together with \eqref{eq:def_A} yield the initial
conditions for $A_{\pm}$:
\[
\begin{cases}
\left.\left(A_+^{2}+A_-^{2}\right)\right|_{t=0}=2\left|p_0\right|\\
\left.\left(A_+^{2}-A_-^{2}\right)\right|_{t=0}=2p_0\\
\left.A_+A_-\right|_{t=0}=0
\end{cases}
\quad \Longleftrightarrow \quad
\begin{cases}
p_0>0\\
\left.A_+\right|_{t=0}=\pm\sqrt{2p_0}\\
\left.A_-\right|_{t=0}=0
\end{cases}
\vee\quad
\begin{cases}
p_0<0\\
\left.A_+\right|_{t=0}=0\\
\left.A_-\right|_{t=0}=\pm\sqrt{-2p_0}
\end{cases}
\]
In what follows assume that $p_0>0$ and $\left.A_+\right|_{t=0}>0$.
Other cases can be treated similarly. Note that then $p_0=\sqrt{2E}$, where $E>0$ is the total energy of the harmonic oscillator, $H=H|_{t=0}=E$.

From \eqref{eq:theorem} we get the following linear system:
\begin{equation}
\label{eq:constants} \left\{
  \begin{array}{lll}
    \m{}_{23}^{1}=C_2p_0-C_4, & \m{}_{31}^{1}=C_3p_0-C_1, & \m{}_{12}^{1}=C_5\sqrt{2p_0}\\
    \m{}_{13}^{2}=C_2p_0+C_4, &
    \m{}_{12}^{2}=-C_6\sqrt{2p_0}, &
    \m{}_{23}^{2}=C_3p_0+C_1\\
    \m{}_{13}^{3}=C_7\sqrt{2p_0}, &
\m{}_{23}^{3}=-C_8\sqrt{2p_0}, & \m{}_{12}^{3}=C_9
\end{array}
\right.
\end{equation}

One can easily check that the unique solution of the latter system
with respect to $C_\nu$ ($\nu=1,\ldots,9$) is
\[
\left\{
  \begin{array}{lll}
C_1=\frac{1}{2}\left(\overset{\circ}{\mu}{}_{23}^{2}-\overset{\circ}{\mu}{}_{31}^{1}\right),&
C_2=\frac{1}{2p_0}\left(\overset{\circ}{\mu}{}_{13}^{2}+\overset{\circ}{\mu}{}_{23}^{1}\right),&
C_3=\frac{1}{2p_0}\left(\overset{\circ}{\mu}{}_{23}^{2}+\overset{\circ}{\mu}{}_{31}^{1}\right)\vspace{1mm}\\
C_4=\frac{1}{2}\left(\overset{\circ}{\mu}{}_{13}^{2}-\overset{\circ}{\mu}{}_{23}^{1}\right),&
C_5=\frac{1}{\sqrt{2p_0}}\overset{\circ}{\mu}{}_{12}^{1},&
C_6=-\frac{1}{\sqrt{2p_0}}\overset{\circ}{\mu}{}_{12}^{2}\vspace{1mm}\\
C_7=\frac{1}{\sqrt{2p_0}}\overset{\circ}{\mu}{}_{13}^{3},&
C_8=-\frac{1}{\sqrt{2p_0}}\overset{\circ}{\mu}{}_{23}^{3},&
C_9=\overset{\circ}{\mu}{}_{12}^{3}
\end{array}
\right.
\]

\begin{rem}
Note that the parameters $C_{\nu}$  have to satisfy condition \eqref{eq:cond} to get the operadic Lax representations.
\end{rem}

\begin{defn}
If $\mu=\overset{\circ}{\mu}$, then the multiplication $\overset{\circ}{\mu}$ is called \emph{dynamically rigid}. 
If $\mu\neq\overset{\circ}{\mu}$, then the multiplication $\mu$ is called a \emph{dynamical deformation of $\overset{\circ}{\mu}$} (over the harmonic oscillator).

\end{defn}

\section{Bianchi classification of 3d real Lie algebras}

We use the Bianchi classification of the 3D real Lie algebras given in
\cite{Landau80}.
The structure equations of the 3D real Lie algebras can be presented
as follows:
\[
[e_1,e_2]=-\alpha e_2+n^{3}e_3,\quad
[e_2,e_3]=n^{1}e_1,\quad
[e_3,e_1]=n^{2}e_2+\alpha e_3
\]
The values of the parameters $\alpha,n^{1}, n^{2},n^{3}$  and the corresponding structure constants are presented in Table \ref{table:Bianchi1}.
\begin{table}[ht]
\begin{center}
\begin{tabular}{|c||c||c|c|c||c|c|c|c|c|c|c|c|c|c|c|}\hline
Bianchi type & $\alpha$ & $n^{1}$ & $n^{2}$ & $n^{3}$ &
$\overset{\circ}{\mu}{}_{12}^{1}$ &
$\overset{\circ}{\mu}{}_{12}^{2}$ &
$\overset{\circ}{\mu}{}_{12}^{3}$ &
 $\overset{\circ}{\mu}{}_{23}^{1}$ & $\overset{\circ}{\mu}{}_{23}^{2}$ & $\overset{\circ}{\mu}{}_{23}^{3}$
  & $\overset{\circ}{\mu}{}_{31}^{1}$ & $\overset{\circ}{\mu}{}_{31}^{2}$ &
  $\overset{\circ}{\mu}{}_{31}^{3}$\\\hline\hline
 I & 0 & 0 & 0 & 0 & 0 & 0 & 0 & 0 & 0 & 0 & 0 & 0 & 0
\\\hline
II & 0 & $1$ & 0 & 0 & 0 & 0 & 0 & $1$ & 0 & 0 & 0 & 0 & 0
\\\hline
VII & 0 & $1$ & $1$ & 0 & 0 & 0 & 0 & $1$ & 0 & 0 & 0 & $1$ & 0
\\\hline
VI & 0 & $1$ & $-1$ & 0 & 0 & 0 & 0 & $1$ & 0 & 0 & 0 & $-1$ & 0
\\\hline
IX & 0 & $1$ & $1$ & $1$ & 0 & 0 & $1$ & $1$ & 0 & 0 & 0 & $1$ & 0
\\\hline
VIII & 0 & $1$ & $1$ & $-1$ & 0 & 0 & $-1$ & $1$ & 0 & 0 & 0 & $1$ &
0
\\\hline
V & 1 & 0 & 0 & 0 & 0 & $-1$ & 0 & 0 & 0 & 0 & 0 & 0 & $1$
\\\hline
IV & 1 & 0 & 0 & 1 & 0 & $-1$ & $1$ & 0 & 0 & 0 & 0 & 0 & $1$
\\\hline
VII$_{a}$ & $a$ & 0 & $1$ & $1$ & 0 & $-a$ & $1$ & 0 & 0 & 0 &
0 & $1$ & $a$
\\\hline
III$_{a=1}$& 1 & 0 & $1$ & $-1$ & 0 & $-1$ & $-1$ & 0 & 0 & 0 & 0 &
$1$ & $1$
\\\hline
VI$_{a\neq 1}$& $a$ & 0 & $1$ & $-1$ & 0 & $-a$ & $-1$ & 0 & 0 & 0 &
0 & $1$ & $a$
\\\hline
\end{tabular}
\end{center}
\caption{3d real Lie algebras in Bianchi classification. Here $a>0$.}
\label{table:Bianchi1}
\end{table}

The Bianchi classification is for instance used to describe the spatially homogeneous spacetimes of dimension 3+1.
In particular, the Lie algebra VII$_a$ is very interesting for the cosmological applications, because it is related to the Friedmann-Robertson-Walker metric.
One can find more details in Refs. \cite{Landau80,GH}. 

\section{Dynamical deformations of 3d real Lie algebras}

By using the structure constants of the 3-dimensional Lie algebras
in the Bianchi classification, Theorem \ref{thm:main} and relations
\eqref{eq:constants} we can find evolution of these algebras generated by the harmonic oscillator (see Table \ref{table:Bianchi3}).

\begin{table}[!h]
\begin{center}\setlength\extrarowheight{4pt}
\begin{tabular}{|c||c|c|c|c|c|c|c|c|c|c|c|}\hline
Deformed Bianchi type & $\mu_{12}^{1}$ & $\mu_{12}^{2}$ &
$\mu_{12}^{3}$ & $\mu_{23}^{1}$ & $\mu_{23}^{2}$ & $\mu_{23}^{3}$ &
$\mu_{31}^{1}$ & $\mu_{31}^{2}$ &  $\mu_{31}^{3}$
\\[1.5ex]\hline\hline
I$^{t}$ & 0 & 0 & 0 & 0 & 0 & 0 & 0 & 0 & 0
\\ [1.5ex] \hline
II$^{t}$ & 0 & 0 & 0 & $\frac{p+p_0}{2p_0}$ & $\frac{\omega
q}{2p_0}$ & 0 & $\frac{\omega q}{2p_0}$ & $\frac{p-p_0}{-2p_0}$ & 0
\\ [1.5ex] \hline
VII$^{t}$ & 0 & 0 & 0 & $1$ & 0 & 0 & 0 & $1$ & 0
\\ [1.5ex] \hline
VI$^{t}$ & 0 & 0 & 0 & $\frac{p}{p_0}$& $\frac{\omega q}{p_0}$ & 0 &
$\frac{\omega q}{p_0}$ & $-\frac{p}{p_0}$ & 0
\\ [1.5ex] \hline
IX$^{t}$ & 0 & 0 & $1$ & $1$ & 0 & 0 & 0 & $1$ & 0
\\ [1.5ex] \hline
VIII$^{t}$ & 0 & 0 & $-1$ & $1$ & 0 & 0 & 0 & $1$ & 0
\\ [1.5ex] \hline
V$^{t}$ & $\frac{A_-}{\sqrt{2p_0}}$ & $\frac{-A_+}{\sqrt{2p_0}}$ & 0
& 0 & 0 & $\frac{-A_-}{\sqrt{2p_0}}$ & 0 & 0 &
$\frac{A_+}{\sqrt{2p_0}}$
\\ [1.5ex] \hline
IV$^{t}$ & $\frac{A_-}{\sqrt{2p_0}}$ & $\frac{-A_+}{\sqrt{2p_0}}$ &
$1$ & 0 & 0 & $\frac{-A_-}{\sqrt{2p_0}}$ & 0 & 0 &
$\frac{A_+}{\sqrt{2p_0}}$
\\ [1.5ex] \hline
VII$^{t}_a$ & $\frac{aA_-}{\sqrt{2p_0}}$ &
$\frac{-aA_+}{\sqrt{2p_0}}$ & $1$ & $\frac{p-p_0}{-2p_0}$ &
$\frac{\omega q}{-2p_0}$ & $\frac{-aA_-}{\sqrt{2p_0}}$ &
$\frac{\omega q}{-2p_0}$ & $\frac{p+p_0}{2p_0}$ &
$\frac{aA_+}{\sqrt{2p_0}}$
\\ [1.5ex] \hline
III$_{a=1}^{t}$ & $\frac{A_-}{\sqrt{2p_0}}$ &
$\frac{-A_+}{\sqrt{2p_0}}$ & $-1$ & $\frac{p-p_0}{-2p_0}$ &
$\frac{\omega q}{-2p_0}$ & $\frac{-A_-}{\sqrt{2p_0}}$ & $\frac{\omega
q}{-2p_0}$ & $\frac{p+p_0}{2p_0}$ & $\frac{A_+}{\sqrt{2p_0}}$
\\ [1.5ex] \hline
VI$_{a\neq1}^{t}$ & $\frac{aA_-}{\sqrt{2p_0}}$ &
$\frac{-aA_+}{\sqrt{2p_0}}$ & $-1$ & $\frac{p-p_0}{-2p_0}$ &
$\frac{\omega q}{-2p_0}$ & $\frac{-aA_-}{\sqrt{2p_0}}$ &
$\frac{\omega q}{-2p_0}$ & $\frac{p+p_0}{2p_0}$ &
$\frac{aA_+}{\sqrt{2p_0}}$
\\ [1.5ex] \hline
\end{tabular}
\end{center}
\caption{Evolution of 3d real Lie algebras. Here $p_0=\sqrt{2E}$.}
\label{table:Bianchi3}
\end{table}

\begin{thm}[dynamically rigid algebras]
The algebras I, VII, VIII, and IX are dynamically rigid over the harmonic oscillator.
\end{thm}
\begin{proof}[Proof]
This is evident from Tables \ref{table:Bianchi1} and \ref{table:Bianchi3}.
\end{proof}
\begin{thm}[dynamical Lie algebras]
\label{thm:lie}
The algebras II$\,^{t}$, IV$\,^{t}$, V$\,^{t}$, VI$\,^{t}$, 
III$_{a=1}^{\,t}$, VI$_{a\neq1}^{\,t}$, and VII$^{\,t}_a$ are Lie algebras.
\end{thm}
\begin{proof}
Denoting $\mu:=[\cdot,\cdot]_t$, one has to calculate the Jacobiator defined by
\begin{align*}
J_t(x;y;z)
&:=[x,[y,z]_t]_t+[y,[z,x]_t]_t+[z,[x,y]_t]_t \\
&\,\,=J^{1}_t(x;y;z)e_1+J^{2}_t(x;y;z)e_2+J^{3}_t(x;y;z)e_3
\end{align*}
At first, by direct calculations one can check that for the algebras
II$^{t}$, IV$^{t}$, V$^{t}$, VI$^{t}$ the Jacobiator identically vanishes, i.e $J_t = 0$.

Denote the scalar triple product of the vectors $x,y,z$ by
\[
(x,y,z)\=
 \begin{vmatrix}
 x^{1} & x^{2} & x^{3} \\
 y^{1} & y^{2} & y^{3} \\
z^{1} & z^{2} & z^{3} \\
\end{vmatrix}
\]
Then, by direct calculations one can check that for the algebras VI$_{a\neq1}^{\,t}$ and  VII$^{\,t}_a$ the Jacobiator coordinates are
\begin{equation}
\label{eq:jacobi}
\left\{
    \begin{array}{ll}
      J^{1}_t(x;y;z)=-\frac{a(x,y,z)}{\sqrt{2p_0^{3}}}\left[A_-\omega q +A_+(p-p_0)\right]\vspace{1mm}\\
      J^{2}_t(x;y;z)=-\frac{a(x,y,z)}{\sqrt{2p_0^{3}}}\left[A_+\omega q - A_-(p+p_0)\right]\vspace{1mm}\\
      J^{3}_t(x;y;z)=0
    \end{array}
  \right.
\end{equation}
and for the algebra III$_{a=1}^{\,t}$ one has the same formulae with $a=1$.

Now, by using relations \eqref{eq:def_A} calculate:
\begin{align*}
 A_-\omega q +A_+(p-p_0)
 &=A_+A^{2}_-+A_+(p-p_0)\\
 &=A_+(A^{2}_- + p -p_0)\\
 &=A_+\left(A^{2}_- + \frac{1}{2}A^{2}_+ - \frac{1}{2}A^{2}_- - p_0 \right)\\
 &=A_+\left(\frac{1}{2}A^{2}_+ + \frac{1}{2}A^{2}_- - p_0 \right)\\
 &=A_+(\sqrt{2H}-p_0)\\
 &\hskip-2mm \overset{H=E}{=}A_+(\sqrt{2E}-p_0)\\
 &=A_+ 0\\
 &= 0
 \end{align*}
Here we used the fact that the Hamiltonian $H$ is a \emph{conserved} observable, i.e
\[
H=H|_{t=0}=E=p^{2}_0/2
\]
Thus, we have proved that $J^{1}_t=0$.
In the same way one can check that $J^{2}_t=0$.
\end{proof}

\section{Energy conservation from Jacobi identities}

When proving Theorem \ref{thm:lie} we observed how the conservation
of energy $H=E$ implies the Jacobi identities $J^{1}_t=J^{2}_t=0$ of
the dynamically deformed algebras. Now let us show \emph{vice versa}, i.e

\begin{thm}
\label{thm:H=E}
The Jacobi identity $J_t=0$ implies conservation of energy $H=E$.
\end{thm}

\begin{proof}
By setting in \eqref{eq:jacobi} $J^{1}_t=J^{2}_t=0$, we obtain the following system:
\begin{equation}
\label{eq_pq}
\begin{cases}
A_-\omega q+A_+p=A_+p_0\\
A_+\omega q-A_-p=A_-p_0
\end{cases}
\end{equation}
Now use the the defining relations \eqref{eq:def_A} of $A_{\pm}$ and
the Cramer formulae to express the canonical variables $q,p$ via
$A_\pm$. First calculate
\begin{align*}
\Delta
&\=
\begin{vmatrix}
A_-&A_+\\
A_+&-A_-
\end{vmatrix}
=-A^{2}_- - A^{2}_+
=-2\sqrt{2H}\neq0
\\
\Delta_{\omega q}
&\=
\begin{vmatrix}
A_+ p_0 & A_+\\
A_- p_0 &-A_-
\end{vmatrix}
=-2A_+A_- p_0
= -2\omega q p_0
\\
\Delta_p
&\=
\begin{vmatrix}
A_- & A_+ p_0\\
A_+ &A_- p_0
\end{vmatrix}
=A^{2}_- p_0-A^{2}_+ p_0
= -2p p_0
\end{align*}
Thus we have
\begin{align*}
\omega q
&= \frac{\Delta_{\omega q}}{\Delta}=-\frac{2\omega q p_0}{-2\sqrt{2H}}
\hskip12.7mm \quad\Longrightarrow\quad
 \frac{p_0}{\sqrt{2H}}=1
\quad\Longrightarrow\quad
H=p^{2}_0/2=E
\\
p
&= \frac{\Delta_{p}}{\Delta}
=-\frac{2 p p_0}{-2\sqrt{2H}}
=\frac{p p_0}{\sqrt{2H}}
\quad\Longrightarrow\quad
 \frac{p_0}{\sqrt{2H}}=1
 \quad\Longrightarrow\quad
H=p^{2}_0/2=E
\end{align*}
Actually, the last implications are possible only at the time moments when $q\neq0$ and  $p\neq0$, respectively. But $q$ and $p$ can not be simultaneously zero, thus really $H=E$ for all $t$.
\end{proof}

\section{Concluding remarks}

In the present paper, the operadic Lax representations for the  harmonic oscillator were  used to construct the dynamical deformations of 3d real Lie algebras in the Bianchi classification.
It was shown that the energy conservation of the harmonic oscillator is related to the Jacobi identities of the dynamically deformed algebras.
Based on this observation, it was proved that the dynamical deformations of 3D real Lie algebras in the Bianchi classification over the harmonic oscillator are Lie algebras.

It turned out that four Lie algebras (I,VII,IX,VIII) remain undeformed (rigid) and all other ones are deformed. However, it is interesting to note that these remain to be Lie algebras over canonical variables $q,p$.  Namely, four of them 
(II$\,^{t}$, IV$\,^{t}$, V$\,^{t}$, VI$\,^{t}$) lead straightforwardly to the Jacobi identity, while in other cases 
(VII$^{\,t}_a$, III$_{a=1}^{\,t}$, VI$_{a\neq1}^{\,t}$) satisfy the Jacobi identity with  the energy conservation law. 

\section*{Acknowledgements}

The research was in part supported by the Estonian Science Foundation, Grant No. ETF-6912. The authors are grateful to S. Hervik and P. Kuusk for discussions about using the Bianchi  classification  in cosmology and to the referee for useful comments.

\medskip
\noindent
Department of Mathematics, Tallinn University of Technology\\
Ehitajate tee 5, 19086 Tallinn, Estonia\\
E-mails: eugen.paal@ttu.ee and jvirkepu@staff.ttu.ee


\begin{thebibliography}{9}
\itemsep-3pt

\bibitem{BBT03}
O.~Babelon, D.~Bernard, and M.~Talon.
\textit{Introduction to Classical Integrable Systems.}
Cambridge Univ. Press, 2003.

\bibitem{Ger}
M.~Gerstenhaber.
The cohomology structure of an associative ring.
Ann. of Math. {\bf 78} (1963), 267-288.

\bibitem{Landau80}
L. Landau and E. Lifshitz.
\textit{Theoretical Physics.} Vol. 2: \textit{Theory of Field.}
Moskva, Nauka, 1973 (in Russian).
English Translation: \textit{Course of Theoretical Physics.} Vol. 2:
\textit{The Classical Theory of Fields.} Butterworth-Heinemann, 1980.

\bibitem{Lax68}
P.~D.~ Lax.
Integrals of nonlinear equations of evolution and solitary waves.
Comm. Pure Applied Math. {\bf21} (1968), 467-490.

\bibitem{GH}
\O. Gr\o n and S. Hervik. Einstein's General Theory of Relativity.
Springer, New York, 2007.

\bibitem{Paal07}
E.~Paal.
Invitation to operadic dynamics.
J. Gen. Lie Theory Appl. {\bf1} (2007), 57-63.

\bibitem{PV07}
E.~Paal and J.~Virkepu.
Note on operadic harmonic oscillator.
Rep. Math. Phys. {\bf 61} (2008), 207-212.

\bibitem{PV08}
E.~Paal and J.~Virkepu.
2D binary operadic Lax representation for harmonic oscillator.
Preprint \texttt{arXiv:0803.0592}, 2008.

\bibitem{PV08-1}
E.~Paal and J.~Virkepu.
Operadic represenatations of harmonic oscillator in some 3d Lie algebras.
J. Gen. Lie Theory Appl. {\bf3} (2009), 53--59.

\end{thebibliography}
\end{document}